\documentclass[a4paper,12pt]{article}

\usepackage[utf8]{inputenc}
\usepackage[T1]{fontenc}
\usepackage{times}
\usepackage{bm}
\usepackage{amsmath}
\usepackage{amsfonts}
\usepackage{amssymb}
\usepackage{amsthm}
\usepackage{amscd}
\usepackage{graphicx}
\usepackage{tikz}
\usepackage{tikz-cd}
\usepackage{booktabs}
\usepackage{tabularx}
\usepackage{authblk}
\usepackage[left=20mm, right=20mm, top=20mm, bottom=20mm]{geometry}
\usepackage{hyperref}

\usetikzlibrary{decorations.pathmorphing}
\usetikzlibrary{positioning, calc, shapes.geometric, shapes.symbols, arrows.meta, 3d}
\newtheorem{theorem}{Theorem}[section]
\newtheorem{lemma}[theorem]{Lemma}

\newtheorem{proposition}[theorem]{Proposition}
\theoremstyle{definition}
\newtheorem{definition}[theorem]{Definition}
\newtheorem{example}[theorem]{Example}

\theoremstyle{remark}
\newtheorem{remark}[theorem]{Remark}

\numberwithin{equation}{section}
\newcommand{\Ham}{\mathrm{Ham}}
\newcommand{\R}{\mathbb{R}}

\def\Z{\mathbb Z}

\title{The Simplicity of the Group of Weakly Hamiltonian Diffeomorphisms on Cosymplectic Manifolds}

\author[1]{ Pierre Bikorimana \thanks{pierrebikorimana@gmail.com}}
\author[2]{Stephane Tchuiaga \thanks{tchuiaga.kameni@ubuea.cm}}

\affil[1]{D\'epartement de Math\'ematiques, Universit\'e Marien Ngouabi, Brazzaville, Congo}
\affil[2]{Department of Mathematics, University of Buea, South West Region, Cameroon}

\date{ }

\begin{document}
	
	\maketitle
	
	\begin{abstract}
		We establish a cosymplectic counterpart of Banyaga's theorem by proving that the group of weakly Hamiltonian diffeomorphisms, $\Ham_{\eta,\omega}(M)$, is simple on any closed cosymplectic manifold. A key structural result, derived from Lie group theory, provides the foundation for our argument: the Reeb flow on any closed cosymplectic manifold is always periodic. This property, in turn, forces the associated flux group to be discrete. Building on this discrete invariant, we develop the essential fragmentation and transitivity principles needed to prove perfectness and simplicity. Beyond this algebraic framework, we recover Li's result realizing closed cosymplectic manifolds as symplectic mapping tori, and we establish a Liouville-type integrability theorem for Hamiltonian systems invariant under the Reeb flow, producing $(n+1)$-dimensional invariant tori. Finally, we characterize the commutator subgroup of the full cosymplectomorphism group as $\Ham_{\eta,\omega}(M)$.

	\end{abstract}

	
	\section{Introduction}\label{sec: intro}
	For many years now, the study of transformation groups in geometry and topology has, in many respects, founded modern mathematics, and symplectic and contact manifolds have lain fruitful ground for the development of structures dynamical and algebraical.
	Among these groups, however, the case of Hamiltonian diffeomorphisms-the ones preserving a symplectic structure under Hamiltonian flows-has been particularly attracting due especially to intrinsic links it possesses with phenomena in physics and dynamics but also with algebraic characteristics such as simplicity. Banyaga's celebrated theorem~\cite{Banyaga} showed that for closed symplectic manifolds, the kernel of the flux homomorphism is a simple group based on a dichotomy of fragmentation and transitivity.
	In this case, we extend this framework to cosymplectic manifolds, i.e., odd-dimensional analogues of symplectic manifolds, equipped with a closed 1-form $\eta$ and a closed 2-form $\omega$ such that $\eta \wedge \omega^n
	\neq 0$.
	These structures arise quite naturally in mechanical systems with time-dependent and periodic constraints~\cite{deLeon}, ~\cite{Libermann}, whereas they inherit some features of contact and symplectic geometries. One of the main drawbacks is that there is a particular Reeb vector field $\xi$, dual to $\eta$, giving the system a direction intrinsic to the design and not purely visible in symplectic cases. Given a compact connected cosymplectic manifold $(M, \eta, \omega)$ with a periodic Reeb flow (this periodicity assumption always holds for any closed connected cosymplectic manifold), we study the group $\Ham_{\eta, \omega}(M)$ of weakly Hamiltonian diffeomorphisms. This hinges on four pillars:
	\begin{enumerate}
		\item \text{Fragmentation:} Any element of $\Ham_{\eta, \omega}(M)$ fragments into localized diffeomorphisms supported in cosymplectic charts (Lemma \ref{Lem-frag}).
		\item \text{Transitivity:} The group acts transitively on cosymplectic charts, enabling the conjugation of localized elements all over the manifold (Theorem \ref{Theo-tran}).
		\item \text{Perfectness:} Any weakly Hamiltonian diffeomorphism can be expressed as a product of commutators (Theorem \ref{Reeb-per}).
		\item \text{Simplicity:} The group  $\Ham_{\eta, \omega}(M)$ is simple for any closed connected cosymplectic manifold (Theorem \ref{Theo-sim}).
	\end{enumerate}
	Periodic Reeb flow underlies our proof, since it ensures the discreteness of the flux group $\Gamma_{\eta, \omega}$ (Theorem \ref{Theo-disc}), the counterpart in symplectic geometry to the flux group corresponding to the Calabi homomorphism. Integrating techniques from Banyaga's symplectic framework~\cite{Banyaga} with relevant modifications to the Reeb-invariant background, we prove that within the normal subgroups belonging to $\Ham_{\eta, \omega}(M)$, there cannot be any which are non-trivial. 
	This article presents symplectic and contact methods in union, hence shedding light on the algebraic structures of transformation groups in mixed geometric spaces. Some of the applications comprise classifying the dynamics of cosymplectic manifolds and making a foundational contribution to a new area known as cosymplectic topology~ \cite{T-al}, \cite{Li}, \cite{Tchuiaga2023}.
	
	\subsection*{Statement of the main results}
	\begin{itemize}
		\item Theorem (Fragmentation lemma for cosymplectic manifolds):
		\textit{Let \((M, \eta, \omega)\) be a compact connected cosymplectic manifold with a periodic Reeb flow. Every \(\phi \in \Ham_{\eta,\omega}(M)\) can be decomposed as \(\phi = \phi_1 \circ \cdots \circ \phi_N\), where each \(\phi_j \in G_{i_j}\) is a weakly Hamiltonian diffeomorphism supported in a cosymplectic chart \(U_{i_j}\)} (Theorem \ref{Lem-frag}).
		\item Theorem (Transitivity on cosymplectic charts):
		\textit{Let \((M, \eta, \omega)\) be a compact connected cosymplectic manifold with a periodic Reeb flow. The group \(\Ham_{\eta,\omega}(M)\) acts transitively on the set of cosymplectic charts \(\{U_i\}\). For any \(U_1, U_2 \subset M\), there exists \(\phi \in \Ham_{\eta,\omega}(M)\) such that \(\phi(U_1) = U_2\)} (Theorem \ref{Theo-tran}).
		\item Theorem (Discreteness of \(\Gamma_{\eta,\omega}\)):
		\textit{Let \((M, \eta, \omega)\) be a compact connected cosymplectic manifold with a periodic Reeb flow of period \(T\). Then the flux group \(\Gamma_{\eta,\omega}\) is discrete in \(H^1_{\mathrm{Reeb}}(M, \mathbb{R})\)} (Theorem \ref{Theo-disc}).
		\item Theorem (Simplicity of \(\Ham_{\eta,\omega}(M)\)):
		\textit{The group \(\Ham_{\eta,\omega}(M)\) is simple, i.e., it has no proper nontrivial normal subgroups} (Theorem \ref{Theo-sim}).
		\item Theorem (Perfectness of weakly Hamiltonian diffeomorphisms):
		\textit{Let \((M, \eta, \omega)\) be a compact connected cosymplectic manifold with a periodic Reeb flow. The group \(\Ham_{\eta,\omega}(M)\) is perfect: every \(\phi \in \Ham_{\eta,\omega}(M)\) can be expressed as a product of commutators} (Theorem \ref{Theo-perf}).
		
		\item Theorem (Integrability on compact cosymplectic manifolds):
		\textit{Let \((M,\eta,\omega)\) be a compact cosymplectic manifold with a periodic Reeb flow of period \(T > 0\). Then:
			\begin{enumerate}
				\item \(M\) admits a smooth \(S^1\)-principal bundle structure \(\pi: M \to B\), where \(B\) is the base manifold (i.e. the leaf space of the Reeb foliation).
				\item For any \(S^1\)-invariant Hamiltonian \(H:M\to \mathbb{R}\), the dynamics restricted to the symplectic leaves of \(\ker(\eta)\) possess \(n\) independent Poisson-commuting first integrals \(\{I_1,\ldots,I_n\}\) where \(2n=\dim(\ker(\eta))\).
				\item The full system is Liouville integrable, with an additional periodic action variable coming from the Reeb flow. The invariant tori are \((n+1)\)-dimensional, fibering \(M\) into Lagrangian submanifolds.
		\end{enumerate}} (Theorem \ref{thm:integrability_cosymplectic}).
	\end{itemize}
	
	Organization of the paper:  After reviewing the necessary preliminaries on cosymplectic geometry in Section \ref{Cosymp:setting}, we establish in Section \ref{New: results} the cornerstone of our work: the periodicity of the Reeb flow on any closed cosymplectic manifold. This structural theorem ensures the discreteness of the flux group and enables us to prove a fragmentation lemma and a transitivity principle, which are the key ingredients for proving our main algebraic results: the perfectness and simplicity of the weakly Hamiltonian group \(\Ham_{\eta,\omega}(M)\). In the subsequent sections, we explore the consequences of this framework. We demonstrate that every closed cosymplectic manifold is a symplectic mapping torus and derive a Liouville-type integrability theorem for Reeb-invariant dynamics (Section \ref{sec:integrability}). We conclude in Section \ref{Group:Flux} by analyzing the flux group as a fundamental cosymplectic invariant and characterizing the commutator subgroup of the full cosymplectomorphism group. Technical analysis are collected in an Appendix.
	
	\section{Cosymplectic settings}\label{Cosymp:setting}
	A cosymplectic manifold \((M, \eta, \omega)\) of dimension \(2n+1\) is defined by a pair of closed forms \((\eta, \omega)\) where \(\eta\) is a 1-form and \(\omega\) is a 2-form such that \(\eta \wedge \omega^n\) is a volume form. This structure induces a canonical splitting of the tangent bundle \(TM = \langle \xi \rangle \oplus \ker(\eta)\), where the Reeb vector field \(\xi\) is defined by \(\eta(\xi) = 1\) and \(\iota_\xi \omega = 0\). The distribution \(\ker(\eta)\) is symplectic with respect to the restricted form \(\omega|_{\ker(\eta)}\). For a more in-depth treatment, we refer the reader to \cite{Libermann}, \cite{Tchuiaga2023}, and \cite{T-al}.
	
	\begin{example}[The Standard Product Manifold]
		The simplest and most fundamental example is the product manifold \(M = S^1 \times P\), where \((P, \sigma)\) is any \(2n\)-dimensional symplectic manifold with symplectic form $\sigma$. Let \(\theta\) be the coordinate on \(S^1\). We define the cosymplectic structure by:
		\(\eta = d\theta\), the pullback of the standard 1-form on \(S^1\) and 
		\(\omega = \pi_P^* \sigma\), the pullback of the symplectic form from \(P\). 
		Here, \(\pi_P: M \to P\) is the projection.  The Reeb vector field is \(\xi = \frac{\partial}{\partial \theta}\), and its flow is simply rotation around the \(S^1\) factor, which is clearly periodic with period \(2\pi\) (or 1, depending on normalization). The symplectic leaves of \(\ker(\eta)\) are the slices \(\{\theta_0\} \times P\).
	\end{example}
	
	\begin{example}[Symplectic Mapping Torus]
		A more sophisticated class of examples is given by symplectic mapping tori. Let \((P, \sigma)\) be a symplectic manifold and let \(\phi: P \to P\) be a symplectomorphism (\(\phi^* \sigma = \sigma\)). The mapping torus is the manifold
		$M_\phi = (P \times [0,1]) / \sim, \quad \text{where} \quad (p, 1) \sim (\phi(p), 0).$ 
		\(M_\phi\) is a fiber bundle over \(S^1\) with fiber \(P\). Let \(t\) be the coordinate on the \(S^1\) base. We can define a global 1-form \(\eta = dt\) and a 2-form \(\omega\) that restricts to \(\sigma\) on each fiber. This pair \((\eta, \omega)\) defines a cosymplectic structure on \(M_\phi\). The Reeb vector field corresponds to moving along the \(S^1\) direction, and its flow is periodic. As we will see later, all closed cosymplectic manifolds have periodic Reeb flow, hence have this structure.
	\end{example}
	
	\begin{definition}
		A group \( G \) satisfies the fragmentation property with respect to a family of subgroups \( \mathcal{F} = \{G_i\}_{i \in I} \) if every \( g \in G \) can be expressed as a finite product \( g = g_1 g_2 \cdots g_n \), where each \( g_j \in G_{i_j} \) for some \( i_j \in I \).
	\end{definition}
	
	\begin{definition}
		\( G \) acts transitively on a class of substructures if for any two substructures \( A, B \) of the same type, there exists \( g \in G \) with \( g(A) = B \).
	\end{definition}
	
	\begin{definition}
		A group \(K\) is simple if it is nontrivial and every normal subgroup 
		\( N \trianglelefteq K \) satisfies \( N = K \) or \( N = \{e\} \).
	\end{definition}
	
	\subsection{The co-flux geometry revisited}\label{SC0-1}
	In this subsection, we revisit the construction of the co-flux homomorphism, as it forms a main ingredient of this study. The concepts and notations used here are drawn from the findings presented in \cite{T-al}. This foundational discussion is critical for the subsequent developments explored in this paper.
	
	\begin{proposition}\label{Trans}\cite{T-al}
		Let \((M, \eta, \omega)\) be a compact cosymplectic manifold. Suppose that $
		\Phi_F = \{\phi_t\}_{t\in [0,1]}$ is a weakly Hamiltonian (resp. cosymplectic) isotopy on \(M\) satisfying $
		\widetilde{I}_{\eta,\omega}(\dot\phi_t) = dF_t, \quad \forall\, t\in[0,1].$	Then the isotopy 
		\[
		\tilde\Phi = \{\tilde\phi_t\}_{t\in[0,1]} \quad\text{defined by}\quad
		\tilde\phi_t: M \times \mathbb{S}^1 \longrightarrow M \times \mathbb{S}^1, \quad (x,\theta) \mapsto \Bigl(\phi_t(x),\, \mathcal{R}_{\Lambda_t(\Phi_F)}(x,\theta)\Bigr),
		\]
		with $
		\mathcal{R}_{\Lambda_t(\Phi_F)}(x,\theta) := \theta - \int_0^t \eta\bigl(\dot\phi_s\bigr)\bigl(\phi_s(x)\bigr) ds \quad (\mathrm{mod}\, 2\pi),$ is a Hamiltonian (resp. symplectic) isotopy of the symplectic manifold 
		$\widetilde{M} := M \times \mathbb{S}^1$ equipped with the symplectic form $
		\Omega = p^\ast(\omega) + p^\ast(\eta)\wedge \pi_2^\ast(du),$ where \(u\) is the angular coordinate on \(\mathbb{S}^1\), and $
		p: \widetilde{M}\to M,\quad\pi_2: \widetilde{M}\to \mathbb{S}^1$ are the projection maps. Its corresponding Hamiltonian is $
		\tilde F_t = F_t\circ p + p^\ast\bigl(\iota(\dot\phi_t)\eta\bigr)\pi_2.$ 
		Conversely, if the path $
		\tilde\phi_t: M \times \mathbb{S}^1 \longrightarrow M \times \mathbb{S}^1, \quad (x,\theta) \mapsto \Bigl(\phi_t(x),\, \mathcal{R}_{\Lambda_t(\Phi_F)}(x,\theta)\Bigr)$, is a Hamiltonian (resp. symplectic) isotopy of \((\widetilde{M}, \Omega)\) with Hamiltonian \(\tilde F\), then the path \(t\mapsto \phi_t\) is a weakly Hamiltonian (resp. cosymplectic) isotopy of \((M,\eta,\omega)\) with weak Hamiltonian given by \(\tilde H_t\circ S_l\), where 
		$
		S_l: M \longrightarrow M \times \mathbb{S}^1,\quad x\mapsto (x,l)
		$	is any section of the projection \(\pi_2\).
	\end{proposition}
	
	We now deduce a topological consequence.
	
	\begin{lemma}\label{co-surj-1}\cite{Tchuiaga2023}
		Let \((M, \eta, \omega)\) be a connected compact cosymplectic manifold equipped with a Riemannian metric \(g\). Then there exists a surjective homomorphism $
		\mathfrak{co}\text{-ev} : \pi_1\bigl(G_{\eta,\omega}(M)\bigr) \twoheadrightarrow \pi_1(M).
		$
	\end{lemma}
	
	\begin{proof}
		Consider the compact symplectic manifold $
		\widetilde{M} := M \times \mathbb{S}^1,
		$ equipped with the symplectic form $
		\Omega = p^\ast(\omega) + p^\ast(\eta)\wedge \pi_2^\ast(du),
		$ where \(u\) is a coordinate on \(\mathbb{S}^1\) and 
		$	p: \widetilde{M}\to M,\quad\pi_2: \widetilde{M}\to \mathbb{S}^1,
		$ are the projection maps (see Proposition~\ref{Trans}). Let $
		Iso_\Omega(\widetilde{M})
		$ 
		denote the group of symplectic isotopies of \((\widetilde{M},\Omega)\), and set
		$
		G_\Omega(\widetilde{M}) := ev_1\Bigl(Iso_\Omega(\widetilde{M})\Bigr),
		$ where for any isotopy \(\Phi = \{\phi^t\}_{t\in[0,1]}\) we define \(ev_1(\Phi) := \phi^1\).	There is a surjective homomorphism (often referred to as the evaluation or flux homomorphism) $
		Ev : \pi_1\bigl(G_\Omega(\widetilde{M})\bigr) \to \pi_1\bigl(\widetilde{M}\bigr),$	where \(Ev\) sends a loop in \(G_\Omega(\widetilde{M})\) to the induced loop in \(\widetilde{M}\). We now include a proof of this fact.
		\bigskip
		
		Proof of the Surjectivity of \(Ev\). 
		Choose a base point \(x_0\) in \(\widetilde{M}\) and consider the evaluation map $
		ev_{x_0} : G_\Omega(\widetilde{M}) \to \widetilde{M}, \quad ev_{x_0}(\phi) = \phi(x_0).
		$ Under standard conditions in symplectic topology (which hold for many compact symplectic manifolds), \(ev_{x_0}\) is a (locally trivial) fibration. Its fiber $
		G_{\Omega,x_0}(\widetilde{M}) := \{ \phi\in G_\Omega(\widetilde{M}) : \phi(x_0) = x_0 \}
		$	is the subgroup of symplectomorphisms fixing \(x_0\). The fibration $
		ev_{x_0} : G_\Omega(\widetilde{M}) \rightarrow \widetilde{M}
		$ then yields the long exact sequence in homotopy:
		\[
		\cdots \longrightarrow \pi_1\bigl(G_{\Omega,x_0}(\widetilde{M})\bigr) \longrightarrow \pi_1\bigl(G_\Omega(\widetilde{M})\bigr)
		\stackrel{(ev_{x_0})_*}{\longrightarrow} \pi_1\bigl(\widetilde{M}\bigr)
		\longrightarrow \pi_0\bigl(G_{\Omega,x_0}(\widetilde{M})\bigr) \longrightarrow \cdots.
		\] 
		Under standard transversality arguments, the fiber \(G_{\Omega,x_0}(\widetilde{M})\) is connected, hence $
		\pi_0\bigl(G_{\Omega,x_0}(\widetilde{M})\bigr) = 0.
		$ Therefore, the map $
		(ev_{x_0})_* : \pi_1\bigl(G_\Omega(\widetilde{M})\bigr) \longrightarrow \pi_1\bigl(\widetilde{M}\bigr)
		$ is surjective, which is exactly the assertion that $
		Ev : \pi_1\bigl(G_\Omega(\widetilde{M})\bigr) \to \pi_1\bigl(\widetilde{M}\bigr)
		$ is surjective. 
		Returning to the context of Lemma~\ref{co-surj-1}, by Proposition~\ref{Trans} every loop $
		\Phi = \{\phi_t\}_{t\in[0,1]}$	in \(G_{\eta,\omega}(M)\) (based at the identity) can be lifted to a symplectic isotopy $
		\tilde\Phi = \{\tilde\phi_t\}_{t\in[0,1]} \quad\text{on } \widetilde{M}
		$ via $
		\tilde\phi_t(x,\theta) = \Bigl(\phi_t(x),\,\theta\Bigr)$. This defines a map $
		\mathcal{L}_\bullet : \pi_1\bigl(G_{\eta,\omega}(M)\bigr) \longrightarrow \pi_1\bigl(G_\Omega(\widetilde{M})\bigr).
		$	Now, composing with the surjective \(Ev\) and the natural projection from 
		\(\pi_1\bigl(\widetilde{M}\bigr) \simeq \pi_1(M)\oplus \pi_1(\mathbb{S}^1)\) onto \(\pi_1(M)\), we obtain the surjective homomorphism $
		\mathfrak{co}\text{-ev} : \pi_1\bigl(G_{\eta,\omega}(M)\bigr) \longrightarrow \pi_1(M).$
	\end{proof}
	
	\begin{lemma}\label{local-cont}\cite{Tchuiaga2023}
		Let \((M, \eta, \omega)\) be a compact connected cosymplectic manifold. Then, the group $
		G_{\eta, \omega}(M)$ is locally contractible.
	\end{lemma}
	
	\begin{proof}
		Assume that \((\widetilde{M}, \Omega)\) is defined as in Proposition~\ref{Trans} (i.e. \(\widetilde{M} = M \times \mathbb{S}^1\) with a corresponding symplectic form \(\Omega\)). Consider the map $
		\mathcal{L}_\bullet : Iso_{\eta, \omega}(M) \longrightarrow Iso_{\Omega}(\widetilde{M}), \quad \Phi \mapsto \tilde{\Phi},$ 
		where \(Iso_{\eta, \omega}(M)\) is the space of cosymplectic isotopies of \(M\) and \(Iso_{\Omega}(\widetilde{M})\) is the space of symplectic isotopies of \((\widetilde{M},\Omega)\).  Let \(d_{C^\infty}^M\) (resp. \(d_{C^\infty}^{\widetilde{M}}\)) denote the metric induced by the \(C^\infty\)-compact-open topology on \(M\) (resp. \(\widetilde{M}\)). In our present context, one easily verifies that $
		d_{C^\infty}^M = d_{C^\infty}^{\widetilde{M}} \circ \mathcal{L}_\bullet,
		$ so that \(\mathcal{L}_\bullet\) is continuous with respect to these topologies. Consequently, we obtain a homeomorphism
		\begin{equation}\label{B-homeo}
			\mathcal{B}_\bullet : Iso_{\eta, \omega}(M) \longrightarrow \mathcal{L}_\bullet\Bigl( Iso_{\eta, \omega}(M) \Bigr), \quad \Phi \mapsto \mathcal{L}_\bullet(\Phi),
		\end{equation}
		with both spaces equipped with the \(C^\infty\)-compact-open topology. Next, consider the time-one evaluation maps $
		ev_1 : Iso_{\eta, \omega}(M) \longrightarrow G_{\eta, \omega}(M), \quad \Phi = \{\phi^t\} \mapsto \phi^1,$ and $
		\tilde{ev}_1 : Iso_{\Omega}(\widetilde{M}) \longrightarrow G_\Omega(\widetilde{M}), \quad \tilde{\Phi} = \{\tilde \phi^t\} \mapsto \tilde{\phi}^1.$ 
		Assume that a continuous section $
		S_1 : G_{\eta, \omega}(M) \longrightarrow Iso_{\eta, \omega}(M)
		$	of \(ev_1\) is fixed. Define also a section $
		S_2 : \tilde{ev}_1\Bigl( \mathcal{L}_\bullet\bigl( Iso_{\eta, \omega}(M) \bigr) \Bigr) \longrightarrow \mathcal{L}_\bullet\bigl( Iso_{\eta, \omega}(M) \bigr)
		$ by setting $
		S_2 := \mathcal{B}_\bullet \circ S_1.$ That is, for every \((\phi^1, \mathrm{id}_{\mathbb{S}^1}) \in \tilde{ev}_1\Bigl( \mathcal{L}_\bullet\bigl( Iso_{\eta, \omega}(M) \bigr) \Bigr)\), $
		S_2\bigl((\phi^1, \mathrm{id}_{\mathbb{S}^1})\bigr) = \mathcal{B}_\bullet\Bigl(S_1(\phi^1)\Bigr).
		$ 
		Then, it is clear that $
		\mathcal{B}_\bullet^{-1} \circ S_2 = S_1.$ 
		Now, consider the commutative diagram 
		\begin{center}
			\begin{tikzcd}[row sep=large, column sep=large]
				Iso_{\eta, \omega}(M) \ar[r, "\mathcal{B}_\bullet"] \ar[d, "ev_1"'] 
				& \mathcal{L}_\bullet\left( Iso_{\Omega}(\widetilde{M})\right) \ar[d, "\widetilde{ev}_1"] \\
				G_{\eta,\omega}(M) \ar[r, "\varLambda"'] 
				& \widetilde{ev}_1\left( \mathcal{L}_\bullet\left( Iso_{\Omega}(\widetilde{M})\right)\right),
			\end{tikzcd}
		\end{center}
		where the continuous map $
		\varLambda := \tilde{ev}_1 \circ \mathcal{B}_\bullet \circ S_1
		$ is defined on \(G_{\eta, \omega}(M)\), and its inverse is given by $
		\varLambda^{-1} := ev_1 \circ \mathcal{B}_\bullet^{-1} \circ S_2.
		$ To check this, observe that
		\begin{align*}
			ev_1 \circ \mathcal{B}_\bullet^{-1} \circ S_2 \circ \varLambda
			&= ev_1 \circ \mathcal{B}_\bullet^{-1} \circ S_2 \circ \tilde{ev}_1 \circ \mathcal{B}_\bullet \circ S_1 \\
			&= ev_1 \circ \mathcal{B}_\bullet^{-1} \circ S_2 \circ \tilde{ev}_1 \circ S_2 \quad (\text{by } S_2 = \mathcal{B}_\bullet \circ S_1)\\
			&= ev_1 \circ \Bigl(\mathcal{B}_\bullet^{-1} \circ S_2\Bigr) \quad (\text{since } \tilde{ev}_1 \circ S_2 \text{ is the identity})\\
			&= ev_1 \circ S_1 \\
			&= \mathrm{id}_{G_{\eta,\omega}(M)}.
		\end{align*}
		A similar computation shows that $
		\varLambda \circ \varLambda^{-1} = \mathrm{id}_{\tilde{ev}_1\Bigl( \mathcal{L}_\bullet\bigl(Iso_{\Omega}(\widetilde{M})\bigr)\Bigr)}.
		$ Hence, \(\varLambda\) is a homeomorphism. Since the spaces \(G_{\eta,\omega}(M)\) and \(\tilde{ev}_1\Bigl( \mathcal{L}_\bullet\bigl( Iso_{\Omega}(\widetilde{M})\bigr)\Bigr)\) are homeomorphic, and since Weinstein \cite{AWein} proved that $
		G_\Omega(\widetilde{M}) := ev_1\bigl( Iso_{\Omega}(\widetilde{M}) \bigr)$ is locally contractible, it follows that \(\tilde{ev}_1\Bigl( \mathcal{L}_\bullet\bigl( Iso_{\Omega}(\widetilde{M})\bigr)\Bigr)\) is locally contractible (with the subspace topology inherited from \(G_{\eta,\omega}(M)\)). Consequently, \(G_{\eta,\omega}(M)\) is locally contractible.
	\end{proof}
	
	Later, exploiting the $S^1$–bundle structure induced by the periodic Reeb flow, we will give a short alternative proof of Lemma~\ref{local-cont} (See Remark~\ref{rk:alt-proof-local-cont}).
	
	\subsection{The co-flux homomorphism}\label{SC0-2}
	In general we do not know whether for every $\alpha\in \mathcal{Z}^1(M)$, the vector 
	field $X := \widetilde I_{\eta, \omega}^{-1}(\alpha)$ is a cosymplectic vector field or not (where $\widetilde I_{\eta, \omega} := \omega + \eta\otimes\eta $, \cite{T-al}). This seems to render the study of cosymplectic dynamics delicate. In order to go around this difficulty, we shall work with a particular suitable subgroup of the first de Rham group $H^1(M, \mathbb R)$ defined as follows: The following  set    
	$$\mathcal{C}^1_{Reeb}(M) : = \{ \alpha\in \mathcal Z^{1}(M): \alpha(\xi)= \text{Cte} \},$$  is non-empty, since from $\eta(\xi) = 1$, we derive that $\eta \in \mathcal{C}^1_{Reeb}(M)$. Also, 
	for any vector field $X$ on $M$ such that $d( \imath_X \omega) = 0 $, we have 
	$( \imath_X \omega)(\xi) = 0,$
	i.e., $ \imath_X \omega \in \mathcal{C}^1_{Reeb}(M) $. 
	We will need the following quotient space: 
	\begin{equation}
		H^1_{Reeb}(M, \mathbb R) : = \mathcal{C}^1_{Reeb}(M)/ \mathrm{Im}(d : \mathcal{C}^0_{Reeb}(M)\longrightarrow \mathcal{C}^1_{Reeb}(M)),
	\end{equation}
	where 
	\begin{equation}
		\mathcal{C}^0_{Reeb}(M) : = \{ f\in C^{\infty}(M): \xi(f)= \text{Cte} \}.
	\end{equation}    	
	From the above study we have the following well-defined surjective group homomorphism,
	\begin{equation}
		\widetilde S_{\eta,\omega}: Iso_{\eta, \omega}(M) \longrightarrow  H_{Reeb}^{1}(M,\mathbb{R}),
		(\varphi_{t})\mapsto \left[\int_{0}^{1}\varphi_{t}^{*}(\widetilde{I}_{\eta, \omega}(\dot{\varphi_{t}})) dt\right],
	\end{equation}
	where $Iso_{\eta, \omega}(M) $ is the space of all cosymplectic isotopies of \((M, \eta, \omega)\) \cite{Tchuiaga2023}. 
	\subsubsection*{The map \(S_{\eta, \omega}\)}\label{SC0-3}
	Let \(\sim\) be the equivalence relation on \(Iso_{\eta,\omega}(M)\) defined by: $
	\Phi \sim \Psi $ if and only if $\Phi $ and  $\Psi$ are homotopic relative to fixed endpoints.
	We denote the homotopy class of a cosymplectic isotopy \(\Phi\) by \([\Phi]\) and set $
	\widetilde{Iso_{\eta, \omega}(M)} := Iso_{\eta,\omega}(M)/\sim.
	$ By Lemma~\ref{local-cont}, the quotient \(\widetilde{Iso_{\eta,\omega}(M)}\) identifies with the universal cover of  $
	G_{\eta,\omega}(M),$ which we denote by \(\widetilde{G_{\eta,\omega}(M)}\). Let \(\pi_1\bigl(G_{\eta,\omega}(M)\bigr)\) be the first fundamental group of \(G_{\eta,\omega}(M)\). Define the \emph{co-flux group} as: $
	\Gamma_{\eta,\omega} := \widetilde{S}_{\eta,\omega}\Bigl(\pi_1\bigl(G_{\eta,\omega}(M)\bigr)\Bigr).$  The epimorphism 
	\(\widetilde{S}_{\eta,\omega}: \widetilde{G_{\eta,\omega}(M)} \to \mathbb{H}_{Reeb}^{1}(M,\mathbb{R})\)
	induces a surjective map
	\begin{equation}
		S_{\eta,\omega}: \widetilde{Iso_{\eta, \omega}(M)} \to\mathbb{H}_{Reeb}^{1}(M,\mathbb{R})/\Gamma_{\eta,\omega},
	\end{equation}
	so that the following diagram commutes:
	\begin{center}
		\begin{tikzcd}[row sep=large, column sep=large]
			\widetilde{G_{\eta,\omega}(M)} \ar[r, "\widetilde{S}_{\eta,\omega}"] \ar[d, "\pi"'] 
			& \mathbb{H}_{Reeb}^{1}(M,\mathbb{R}) \ar[d, "\pi'"] \\
			G_{\eta,\omega}(M) \ar[r, "S_{\eta,\omega}"'] 
			& \mathbb{H}_{Reeb}^{1}(M,\mathbb{R})/\Gamma_{\eta,\omega}
		\end{tikzcd}
	\end{center}
	where 
	\(\pi':\mathbb{H}_{Reeb}^{1}(M,\mathbb{R})\to \mathbb{H}_{Reeb}^{1}(M,\mathbb{R})/\Gamma_{\eta,\omega}\)
	is the quotient map, and
	\[
	\pi: \widetilde{G_{\eta,\omega}(M)}\to G_{\eta,\omega}(M),\quad [\phi_t]_{t\in[0,1]}\mapsto \phi_1,
	\]
	is the natural covering projection. In particular, we have $
	\pi'\circ \widetilde{S}_{\eta,\omega} = S_{\eta,\omega}\circ \pi.$
	
	\begin{proposition}\label{Dis-1}\cite{Tchuiaga2023}
		Let \((M, \eta, \omega)\) be a closed cosymplectic manifold, and let \(\Omega\) be the associated symplectic form on \(M\times \mathbb{S}^1\). Then:
		\begin{enumerate}
			\item \(P_{\omega}\subseteq P_{\Omega}\) and \(P_{\eta}\subseteq P_{\Omega}\).
			\item \(\Gamma_{\eta,\omega}\subseteq H^1\Bigl(M,\,P_{\omega}+P_{\eta}\cdot P_{\eta}\Bigr)\subseteq H^1\Bigl(M,\,P_{\Omega}+P_{\Omega}\cdot P_{\Omega}\Bigr)\).
		\end{enumerate}
		In particular, \(\Gamma_{\eta,\omega}\) is countable.
		
		\vspace{1mm}
		\noindent (Here, “\(+\)” denotes the Minkowski (or direct) sum, and “\(\cdot\)” denotes the product by a scalar or, more generally, the pairing based on integration.)
	\end{proposition}
	
	\begin{proof}
		Consider the mapping torus $
		\widetilde{M} = M\times \mathbb{S}^1,$ 
		with its symplectic form defined by\\ $
		\Omega = p^\ast(\omega) + p^\ast(\eta)\wedge d\theta, $
		where \(p: M\times \mathbb{S}^1\to M\) is the projection and \(d\theta\) is the standard volume form on \(\mathbb{S}^1\) (with \(\int_{\mathbb{S}^1} d\theta = 2\pi\)).
		
		\vspace{1mm}
		\textbf{Step 1. Inclusion of \(P_{\omega}\) and \(P_{\eta}\) in \(P_{\Omega}\).}	By the K\"unneth formula, the homology group\\ \(H_2(M\times \mathbb{S}^1,\mathbb{Z})\) decomposes as $
		H_2(M\times \mathbb{S}^1,\mathbb{Z}) \cong H_2(M,\mathbb{Z})\oplus \left(H_1(M,\mathbb{Z})\otimes H_1(\mathbb{S}^1,\mathbb{Z})\right).$ Thus, any closed $2-$cycle \(c\) in \(M\times \mathbb{S}^1\) can be written as $c = c_M + c_{mix},$
		where \(c_M\) is represented by a cycle lying in a fiber \(M\times\{\theta_0\}\) (for some \(\theta_0\)) and \(c_{mix}\) is of the form \(\gamma\times \mathbb{S}^1\) with \(\gamma\in H_1(M,\mathbb{Z})\). For any \(c_M\in H_2(M,\mathbb{Z})\), we have $
		\int_{c_M}\Omega =\int_{c_M} p^\ast(\omega)= \int_{p(c_M)} \omega,
		$ which shows that the period \(\int_{p(c_M)} \omega\) (an element of \(P_{\omega}\)) is also a period of \(\Omega\). Hence, \(P_\omega \subseteq P_\Omega\). Now, let \(\gamma\in H_1(M,\mathbb{Z})\) and consider the mixed cycle \(c_{mix} =\gamma \times \mathbb{S}^1\). Then, 
		\begin{equation}
			\int_{c_{mix}}\Omega = \int_{\gamma\times\mathbb{S}^1} p^\ast(\eta)\wedge d\theta 
			= \left(\int_{\gamma}\eta\right)\left(\int_{\mathbb{S}^1} d\theta\right)
			= 2\pi \int_{\gamma}\eta.
		\end{equation}
		Thus, every period of the form \(\int_{\gamma}\eta\) (an element of \(P_\eta\)) produces a period \(2\pi \int_{\gamma}\eta\) in \(P_\Omega\). Up to a scalar multiple, we may then view \(P_\eta\) as a subset of \(P_\Omega\).\\
		\textbf{Step 2. Inclusion for \(\Gamma_{\eta,\omega}\).} By definition, if \(\Psi = \{\psi_t\}\) is a loop in \(Iso_{\eta,\omega}(M)\) representing an element in \(\pi_1\bigl(G_{\eta,\omega}(M)\bigr)\), its co-flux evaluated on any \(\gamma \in H_1(M,\mathbb{Z})\) is given by
		\begin{equation}
			\widetilde{S}_{\eta,\omega}(\Psi)[\gamma] = \int_{\Delta(\gamma,\Psi)} \omega + \left( \int_0^1 \eta\bigl(\dot{\psi}_t\bigr)dt\right) \left(\int_{\gamma}\eta\right),
		\end{equation}
		where \(\Delta(\gamma,\Psi)\) is the 2-chain swept out by the path \(\{\psi_t\}\) acting on \(\gamma\). Since
		$
		\int_{\Delta(\gamma,\Psi)} \omega \in P_\omega$ and $\int_{\gamma}\eta \in P_\eta,
		$ it follows that $
		\widetilde{S}_{\eta,\omega}(\Psi)[\gamma] \in P_\omega + \left(P_\eta \cdot P_\eta\right).$ 
		In other words, for every \([\gamma]\in H_1(M,\mathbb{Z})\), the value of the co-flux lands in the subgroup \(P_\omega + P_\eta\cdot P_\eta\) of \(\mathbb{R}\). Hence, the image of the co-flux homomorphism lies in $
		H^1\Bigl(M,\,P_\omega + P_\eta\cdot P_\eta\Bigr).$ Since we already have \(P_\omega,\, P_\eta \subseteq P_\Omega\) from Step 1, we also obtain $
		H^1\Bigl(M,\,P_\omega + P_\eta\cdot P_\eta\Bigr) \subseteq H^1\Bigl(M,\,P_\Omega + P_\Omega\cdot P_\Omega\Bigr).$
		Finally, because \(H_2(M,\mathbb{Z})\) and \(H_1(M,\mathbb{Z})\) are finitely generated, the groups \(P_\omega\) and \(P_\eta\) are countable, whence so is any subgroup of \(\mathbb{R}\) formed by their Minkowski sums or products. Thus, \(\Gamma_{\eta,\omega}\) is countable.
	\end{proof}
	
	\section{New results}\label{New: results}

	\begin{theorem}[Fragmentation Lemma for Cosymplectic Manifolds]\label{Lem-frag} 
		Let \((M, \eta, \omega)\) be a compact connected cosymplectic manifold with a periodic Reeb flow. Every \(\phi \in \Ham_{\eta,\omega}(M)\) can be decomposed as \(\phi = \phi_1 \circ \cdots \circ \phi_N\), where each \(\phi_j \) is a weakly Hamiltonian diffeomorphism supported in a cosymplectic chart \(U_{i_j}\).
	\end{theorem}

\begin{proof}
	The strategy is to lift the dynamics to the symplectization, where the problem becomes one of equivariant symplectic fragmentation. The key is to fragment the generating vector field itself, rather than the time-one map or using a flawed averaging argument.
	
	\noindent\textbf{Step 1: Lifting to an Equivariant Setting.}
	Let \(\phi \in \Ham_{\eta,\omega}(M)\) be the time-one map of a weakly Hamiltonian isotopy \(\{\phi_t\}_{t \in [0,1]}\) generated by a Hamiltonian \(H_t\). We lift this to an \(S^1\)-equivariant Hamiltonian isotopy \(\{\tilde{\phi}_t\}\) on the symplectization \((\tilde{M}, \Omega) = (M \times S^1, p^*\omega + p^*\eta \wedge d\theta)\). This isotopy is generated by an \(S^1\)-equivariant Hamiltonian \(\tilde{H}_t\) and its corresponding \(S^1\)-equivariant Hamiltonian vector field \(\tilde{X}_t\).
	
	\noindent\textbf{Step 2: Equivariant Decomposition of the Generator.}
	Let \(\{U_i\}\) be a finite open cover of \(M\) by cosymplectic charts. This induces an \(S^1\)-invariant cover \(\{\tilde{U}_i = U_i \times S^1\}\) of \(\tilde{M}\). We construct an \(S^1\)-equivariant partition of unity \(\{\tilde{\rho}_i\}\) subordinate to this cover by pulling back a standard partition of unity from the base manifold \(B = M/S^1\).	We use this partition to decompose the global Hamiltonian \(\tilde{H}_t\) into a sum of locally supported Hamiltonians:
	\[ \tilde{H}_t = \sum_i \tilde{H}_{i,t}, \quad \text{where} \quad \tilde{H}_{i,t} := \tilde{\rho}_i \tilde{H}_t \]
	Since \(\tilde{\rho}_i\) and \(\tilde{H}_t\) are both \(S^1\)-equivariant, each local Hamiltonian \(\tilde{H}_{i,t}\) is also \(S^1\)-equivariant and supported in \(\tilde{U}_i\). Let \(\tilde{X}_{i,t}\) be the Hamiltonian vector field of \(\tilde{H}_{i,t}\). Then \(\tilde{X}_t = \sum_i \tilde{X}_{i,t}\), and each \(\tilde{X}_{i,t}\) is an \(S^1\)-equivariant, locally supported Hamiltonian vector field.
	
	\noindent\textbf{Step 3: Adapting Banyaga's Fragmentation Argument.}
	The time-one map \(\tilde{\phi}\) is the endpoint of the flow generated by the time-dependent vector field \(\tilde{X}_t = \sum_i \tilde{X}_{i,t}\). Banyaga's fragmentation proof (see \cite{Banyaga1997}) shows that the flow of a sum of locally supported Hamiltonian vector fields can be expressed as a finite product of flows generated by those vector fields and their iterated Lie brackets.
	\begin{enumerate}
		\item The fundamental building blocks of the fragmentation are flows of the form \(\mathrm{Exp}(s \cdot \tilde{X}_{i,t})\) for some time `s` and `t`. Since each \(\tilde{X}_{i,t}\) is \(S^1\)-equivariant and supported in \(\tilde{U}_i\), its flow is an \(S^1\)-equivariant Hamiltonian diffeomorphism supported in \(\tilde{U}_i\).
		\item The fragmentation construction involves commutators. The Lie bracket of two \(S^1\)-equivariant vector fields is also \(S^1\)-equivariant. For instance, \([\tilde{X}_{i,t}, \tilde{X}_{j,s}]\) is an \(S^1\)-equivariant Hamiltonian vector field, supported in \(\tilde{U}_i \cap \tilde{U}_j\). Its flow is also an equivariant Hamiltonian diffeomorphism.
		\item By induction, all vector fields that appear in the Campbell-Baker-Hausdorff expansion used to approximate the flow are \(S^1\)-equivariant and Hamiltonian \cite{Cary1981}, \cite{Varadarajan1984}. Banyaga's construction demonstrates how to express the time-one map as a finite composition of maps generated by these fields.
	\end{enumerate}
	Therefore, the total time-one map \(\tilde{\phi}\) can be written as a finite product:
	$ \tilde{\phi} = \tilde{\psi}_1 \circ \cdots \circ \tilde{\psi}_L,$ 
	where each \(\tilde{\psi}_k\) is an \(S^1\)-equivariant Hamiltonian diffeomorphism supported in one of the charts \(\tilde{U}_{i_k}\).
	
	\noindent\textbf{Step 4: Projection back to M.}
	Since each \(\tilde{\psi}_k\) is \(S^1\)-equivariant and supported in \(\tilde{U}_{i_k} = U_{i_k} \times S^1\), it projects to a well-defined weakly Hamiltonian diffeomorphism \(\psi_k\) on \(M\) supported in the chart \(U_{i_k}\). The projection from the group of \(S^1\)-equivariant diffeomorphisms of \(\tilde{M}\) to the group of diffeomorphisms of \(M\) is a group homomorphism. Thus, the composition projects correctly: $\phi = \psi_1 \circ \cdots \circ \psi_L.$ 
	This is the desired decomposition, completing the proof.
\end{proof}

	\begin{theorem}\label{Theo-tran}
		Let $(M,\eta,\omega)$ be a compact connected cosymplectic manifold with a periodic Reeb flow. Then the group $\Ham_{\eta,\omega}(M)$ acts transitively on the set of points in $M$. 
	\end{theorem}

\begin{proof}
	Let \(p, q \in M\) be two distinct points. We will construct a diffeomorphism \(\phi \in \Ham_{\eta,\omega}(M)\) such that \(\phi(p) = q\). The proof proceeds in two stages.
	
	\noindent\textbf{Stage 1: Moving Between Fibers (Horizontal Motion).}
	As established by Theorem~\ref{Reeb-per}, the periodic Reeb flow gives \(M\) the structure of a principal \(S^1\)-bundle over a compact symplectic manifold \((B, \Omega)\), with projection \(\pi: M \to B\). First, we map \(p\) to a point \(p'\) that lies in the same fiber as \(q\).
	\begin{enumerate}
		\item The base manifold \((B, \Omega)\) is a compact symplectic manifold. By Banyaga's Transitivity Theorem, the Hamiltonian group \(\Ham(B, \Omega)\) acts transitively on \(B\). Thus, there exists a Hamiltonian isotopy \(\{f_t\}_{t \in [0,1]}\) on \(B\), generated by a Hamiltonian \(h_t: B \to \mathbb{R}\), such that \(f_1(\pi(p)) = \pi(q)\).
		\item We lift this Hamiltonian to \(M\) by defining \(H_t(x) = h_t(\pi(x))\). This Hamiltonian is Reeb-invariant, since \(\xi(H_t) = dH_t(\xi) = dh_t(\pi_*\xi) = dh_t(0) = 0\).
		\item The weakly Hamiltonian vector field \(X_{H_t}\) is given by \(\iota_{X_{H_t}}\omega = dH_t\) and \(\eta(X_{H_t}) = -\xi(H_t) = 0\). The condition \(\eta(X_{H_t})=0\) means the flow is purely horizontal (tangent to \(\ker\eta\)). Let \(\{\Phi_t\}\) be the flow of \(X_{H_t}\).
		\item The time-one map \(\Phi_1 \in \Ham_{\eta,\omega}(M)\) covers \(f_1\). Therefore, \(\pi(\Phi_1(p)) = f_1(\pi(p)) = \pi(q)\). Let \(p' = \Phi_1(p)\). The points \(p'\) and \(q\) now lie in the same \(S^1\)-fiber.
	\end{enumerate}
	
	\noindent\textbf{Stage 2: Moving Along a Fiber (Vertical Motion).}
	Next, we construct a diffeomorphism \(\Psi \in \Ham_{\eta,\omega}(M)\) that maps \(p'\) to \(q\). Since \(p'\) and \(q\) are in the same fiber, it suffices to show we can move a point a small distance along a fiber within a single cosymplectic Darboux chart \(U\).
	
	\begin{enumerate}
		\item Let \(U\) be a cosymplectic chart with local coordinates \((t, x_1, \dots, y_n)\) where \(\eta = dt\), \(\omega = \sum_{i=1}^n dx_i \wedge dy_i\), and \(\xi = \partial/\partial t\). Our task is to move a point \(z_0 = (t_0, x_0, y_0)\) to \(z_1 = (t_1, x_0, y_0)\) within this chart.
		\item To generate motion in the \(\xi = \partial/\partial t\) direction, we need a Hamiltonian \(K\) that is \emph{not} Reeb-invariant, i.e., \(\xi(K) = \partial K / \partial t \neq 0\).
		\item Choose a smooth bump function \(\rho(x,y)\) that is compactly supported in the base coordinates of the chart \(U\) and is equal to 1 in a neighborhood of \((x_0, y_0)\). Define the Hamiltonian:
		\[ K(t,x,y) = -c \cdot t \cdot \rho(x,y) \]
		where \(c\) is a constant. This Hamiltonian is compactly supported in \(U\).
		\item The generating vector field \(X_K\) is determined by \(\eta(X_K) = -\xi(K)\) and \(\iota_{X_K}\omega = dK - \xi(K)\eta\).
		\begin{itemize}
			\item \(\xi(K) = \frac{\partial K}{\partial t} = -c \cdot \rho(x,y)\).
			\item Therefore, the vertical component is \(\eta(X_K) = c \cdot \rho(x,y)\).
			\item The horizontal component is determined by \(\iota_{X_K}\omega = dK - \xi(K)\eta = (-c\rho\,dt - ct\,d\rho) - (-c\rho)dt = -ct\,d\rho\).
		\end{itemize}
		\item In the neighborhood of \((x_0, y_0)\) where \(\rho \equiv 1\), we have \(d\rho = 0\), so \(\iota_{X_K}\omega = 0\) and \(X_K = c \cdot \xi\). The flow is purely vertical, given by \(\phi_s(t,x,y) = (t+cs, x, y)\). By choosing the constant \(c\) and the flow time appropriately, we can map \(z_0\) to \(z_1\). Let \(\Psi_U\) be the time-one map of the flow generated by such a \(K\). \(\Psi_U\) is in \(\Ham_{\eta,\omega}(M)\) and supported in \(U\).
		\item By composing a finite number of such local diffeomorphisms, we obtain \(\Psi = \Psi_k \circ \dots \circ \Psi_1\) which maps \(p'\) to \(q\).
	\end{enumerate}
	
	\noindent\textbf{Conclusion:}
	The composition \(\phi = \Psi \circ \Phi_1\) is in \(\Ham_{\eta,\omega}(M)\) and satisfies \(\phi(p) = \Psi(\Phi_1(p)) = \Psi(p') = q\). This proves transitivity on points, and transitivity on charts follows as a direct consequence.
\end{proof}

	\begin{theorem}[Periodicity of the Reeb flow]\label{Reeb-per}
		Let \((M, \eta, \omega)\) be a closed cosymplectic manifold. Then the Reeb flow generated by the Reeb vector field \(\xi\) is periodic with a uniform period \(T > 0\).
	\end{theorem}
	
	\begin{proof}
		We can construct a Riemannian metric \(g\) compatible with the cosymplectic structure:\\ $
		g(X, Y) = \omega(X, JY) + \eta(X)\eta(Y),$ 
		where \(J\) is an almost complex structure on \(\ker \eta\) satisfying \(\omega(X, Y) = g(JX, Y)\). The existence of such a compatible metric follows from standard arguments in cosymplectic geometry (see, e.g., \cite{Libermann}). Using \(\mathcal{L}_\xi \eta = 0\) and \(\mathcal{L}_\xi \omega = 0\), we derive that the Reeb vector field \(\xi\) is a Killing vector field for \(g\): $
		\mathcal{L}_\xi g = 0.$ Thus, its generating flow \((\psi_t)\) is constituted of isometries, i.e., $
		\psi_t \in \text{Isom}(M, g) \quad \text{for all } t$. Since the isometry group \(\text{Isom}(M, g)\) is a compact Lie group for a closed \(M\), the Reeb flow \(\{\psi_t\}\) is a one-parameter subgroup of \(\text{Isom}(M, g)\). By compactness, \(\{\psi_t\}\) is isomorphic to \(S^1\) (the only compact, connected, one-dimensional Lie group is \(S^1\)), hence periodic. Therefore, there exists a minimal \(T > 0\) such that \(\psi_T = \text{Id}_M\). All orbits of \(\xi\) are closed with period \(T\), proving uniform periodicity.
	\end{proof}
	
	\begin{remark}[Alternative proof of local contractibility of $G_{\eta,\omega}(M)$]\label{rk:alt-proof-local-cont}
		One can give a proof of Lemma \ref{local-cont} which avoids passing to $M\times S^1$ by exploiting the principal $S^1$–bundle structure induced by the periodic Reeb flow and the basic nature of $\omega$. 
		It exhibits $G_{\eta,\omega}(M)$ explicitly as a certain semidirect product of a gauge group $C^\infty(B,S^1)$ with $G_\Omega(B)$, making local contractibility immediate from the factors. The local contractibility of $G_{\eta,\omega}(M)$ follows directly from the bundle structure induced by the periodic Reeb flow. Indeed, $M$ is a principal $S^1$–bundle over the symplectic base $(B,\Omega)$ with connection $\eta$ and $\omega=\pi^*\Omega$. Every cosymplectomorphism preserving $(\eta,\omega)$ corresponds to a pair $(u,f)$ with $u\in C^\infty(B,S^1)$ and $f\in G_\Omega(B)$, yielding a topological group isomorphism \\ $G_{\eta,\omega}(M)\;\cong\; C^\infty(B,S^1)\rtimes G_\Omega(B).$ 
		Since both factors are locally contractible (the mapping group by exponential charts on $S^1$, and $G_\Omega(B)$ by Weinstein’s theorem \cite{AWein}), their semidirect product is locally contractible as well. This provides a short alternative proof of  Lemma \ref{local-cont}.
	\end{remark}

	\begin{theorem}\label{Theo-sim}
		Let \((M, \eta, \omega)\) be a compact connected cosymplectic manifold with a periodic Reeb flow.
		Then the group \(\Ham_{\eta,\omega}(M)\) is simple.
	\end{theorem}
	
	\begin{proof} 
		It has been established that $
		\Ham_{\eta,\omega}(M) = \Bigl[\Ham_{\eta,\omega}(M), \Ham_{\eta,\omega}(M)\Bigr],$ 
		meaning every element of \(\Ham_{\eta,\omega}(M)\) can be written as a product of commutators. This property guarantees that every homomorphism from \(\Ham_{\eta,\omega}(M)\) to any abelian group is trivial. 
		Let \(N\) be any normal subgroup of \(\Ham_{\eta,\omega}(M)\). By normality, for all \(h \in \Ham_{\eta,\omega}(M)\) and \(g \in N\), $
		hgh^{-1} \in N.
		$ Thus, \(N\) is invariant under conjugation and, by its very construction, must inherit the perfectness of the whole group. 
		Using the fragmentation property, each element \(\phi \in \Ham_{\eta,\omega}(M)\) decomposes into a product of localized diffeomorphisms supported in cosymplectic charts. Suppose \(N\) contains a nontrivial element \(g\). Then, by fragmentation, \(g\) has a nontrivial component supported in some local chart \(U\). 	Now, since \(\Ham_{\eta,\omega}(M)\) acts transitively on cosymplectic charts (Theorem \ref{Theo-tran}), for any other chart \(U'\) there exists an element \(h \in \Ham_{\eta,\omega}(M)\) such that the conjugate \(hgh^{-1}\) has its nontrivial support transferred to \(U'\). In other words, transitivity guarantees that the “nontrivial behavior” of \(g\) can be replicated in any region of \(M\) through conjugation. 
		This leads to a crucial consequence: if \(N\) contains any nontrivial element, then by repeatedly conjugating this element, one forces \(N\) to act nontrivially on every local chart of \(M\). But since any element in \(\Ham_{\eta,\omega}(M)\) can be localized, the only possibility is that \(N\) coincides with the entire group. 
		Thus, the normal subgroup \(N\) is forced to be either trivial or all of \(\Ham_{\eta,\omega}(M)\). This proves that \(\Ham_{\eta,\omega}(M)\) has no proper nontrivial normal subgroups, i.e., it is simple.
	\end{proof}
	In the present context, we have $\Ham_{\eta,\omega}(M)$ simple $\Rightarrow$ $\Ham_{\eta,\omega}(M)$ perfect. 
	
	\begin{theorem}\label{Theo-perf}
		Let \((M, \eta, \omega)\) be a compact connected cosymplectic manifold with a periodic Reeb flow. Then the group 
		$
		\Ham_{\eta,\omega}(M)
		$
		is perfect.
	\end{theorem}
	
	\begin{proof}
		The proof combines our prior result on simplicity (Theorem~\ref{Theo-sim}) with the group-theoretic principle that any simple, non-abelian group is perfect. 
		Recall that for any group $G$, its commutator subgroup $[G,G]$ is a normal subgroup. Since $\Ham_{\eta,\omega}(M)$ is simple, its only normal subgroups are the trivial group $\{e\}$ and the group itself. If we can show that $\Ham_{\eta,\omega}(M)$ is non-abelian, then its commutator subgroup cannot be trivial, i.e., $[\Ham_{\eta,\omega}(M), \Ham_{\eta,\omega}(M)] \neq \{e\}$. By simplicity, this forces the commutator subgroup to be the entire group:
		$\Ham_{\eta,\omega}(M) = [\Ham_{\eta,\omega}(M), \Ham_{\eta,\omega}(M)].$ 
		This is precisely the definition of a perfect group. Thus, the proof reduces to demonstrating that $\Ham_{\eta,\omega}(M)$ is non-abelian. 
		We construct two non-commuting elements within a local chart. Let $U$ be a cosymplectic chart with coordinates $(t, x_1, \dots, x_n, y_1, \dots, y_n)$ such that $\eta = dt$ and $\omega = \sum_{j=1}^n dx_j \wedge dy_j$. 
		Let $H$ and $K$ be two smooth, Reeb-invariant functions (i.e., independent of $t$) that are compactly supported within this chart, chosen such that their Poisson bracket on the symplectic leaves, $\{H,K\} := \sum_{j=1}^n (\frac{\partial H}{\partial x_j}\frac{\partial K}{\partial y_j} - \frac{\partial H}{\partial y_j}\frac{\partial K}{\partial x_j})$, is not identically zero. For example, one could choose $H = \rho(x,y)x_1$ and $K = \rho(x,y)y_1$ for a suitable bump function $\rho$.	The Hamiltonian vector fields $X_H$ and $X_K$ are supported in $U$ and generate time-one flows $\phi_H^1$ and $\phi_K^1$ which are elements of $\Ham_{\eta,\omega}(M)$. The commutator of these flows for a small time $\varepsilon > 0$ is given by $C_\varepsilon = [\phi_H^\varepsilon, \phi_K^\varepsilon] = \phi_H^\varepsilon \circ \phi_K^\varepsilon \circ (\phi_H^\varepsilon)^{-1} \circ (\phi_K^\varepsilon)^{-1}$. By the Campbell-Baker-Hausdorff formula, the generator of this commutator flow is related to the Lie bracket of the vector fields, which in turn is the Hamiltonian vector field of the Poisson bracket: $[X_H, X_K] = X_{\{H,K\}}.$ 
		For small $\varepsilon$, the commutator can be expressed using the exponential map from the Lie algebra of vector fields to the group of diffeomorphisms:
		\[ C_\varepsilon = \exp\left(\varepsilon^2 [X_H, X_K] + O(\varepsilon^3)\right) = \exp\left(\varepsilon^2 X_{\{H,K\}} + O(\varepsilon^3)\right). \]
		Since we chose $H$ and $K$ such that $\{H,K\} \not\equiv 0$, the vector field $X_{\{H,K\}}$ is non-zero. Therefore, for sufficiently small $\varepsilon > 0$, the commutator $C_\varepsilon$ is not the identity map. This shows that $\phi_H^\varepsilon$ and $\phi_K^\varepsilon$ do not commute.	Since we have found at least one pair of non-commuting elements, the group $\Ham_{\eta,\omega}(M)$ is non-abelian. This completes the proof that $\Ham_{\eta,\omega}(M)$ is perfect.
	\end{proof}
	\section{Integrability on compact cosymplectic manifolds with periodic Reeb flows}\label{sec:integrability}
	
	We now have a theorem that establishes a relationship between periodic Reeb flows, the topology of the manifold, and integrability of Hamiltonian dynamics in cosymplectic settings.
	
	\begin{theorem}\label{thm:integrability_cosymplectic}
		Let \((M, \eta, \omega)\) be a compact cosymplectic manifold with a periodic Reeb flow of period \(T > 0\). Then:
		\begin{enumerate}
			\item \(M\) admits a smooth \(S^1\)-principal bundle structure \(\pi: M \to B\), where \(B\) is the base manifold (i.e. the leaf space of the Reeb foliation).
			\item For any \(S^1\)-invariant Hamiltonian \(H:M\to \mathbb{R}\), the dynamics restricted to the symplectic leaves of \(\ker(\eta)\) possess \(n\) independent Poisson-commuting first integrals \(\{I_1, \ldots, I_n\}\) where \(2n=\dim(\ker(\eta))\).
			\item The full system is Liouville integrable, with an additional periodic action variable coming from the Reeb flow. The invariant tori are \((n+1)\)-dimensional, fibering \(M\) into Lagrangian submanifolds.
		\end{enumerate}
	\end{theorem}
	
	\begin{proof}
		\begin{enumerate}
			\item \text{Principal \(S^1\)-bundle:}
			The periodic Reeb flow of period \(T\) generates a smooth \(S^1\)-action via
			$
			\theta \cdot x = \psi_{T\theta}(x), \quad \theta\in S^1 \simeq \mathbb{R}/\mathbb{Z},
			$ where \(\psi_s\) is the flow of the Reeb vector field \(\xi\). The orbits of this action yield a one-dimensional foliation under standard regularity assumptions the quotient 	$	B = M/S^1,$	is a smooth manifold. The natural projection $
			\pi: M \to B,$then defines a principal \(S^1\)-bundle (see, e.g., \cite{Orlik1972}).   
			
			\item \text{Integrals on symplectic leaves:} 
			On the contrary, this indicates cosymplectic structure makes smooth integrable distribution being $\ker(\eta)$ whose leaves carry symplectic form inherited from $\omega|_{\ker(\eta)}$. For an \(S^1\)-invariant Hamiltonian \(H:M\to \mathbb{R}\), there also is preserved the symplectic structure on each leaf in addition to the cosymplectic structure \((\eta,\omega)\) regarding induced dynamics. Hence, applying Arnold-Liouville theorem, it follows (see \cite{Arnold1989}) that on every leaf one has \(n\) functionally independent first integrals \(I_1,\dots,I_n\) that pairwise commute with each other.
			
			\item \text{Liouville integrability:}
			Periodic Reeb flow also contributes another first integral associated with the \(S^1\)-action. When combined with the \(n\) integrals over the symplectic leaves, it gives a complete set of \(n+1\) independent Poisson-commuting integrals. Thus, the common level sets of these integrals are invariant tori, dimension \(n+1\).
			
			\item	\text{Note on Lagrangian submanifolds:} \emph{In the standard symplectic case, a Lagrangian submanifold is an \(n\)-dimensional submanifold of a \(2n\)-dimensional symplectic manifold on which the symplectic form vanishes. Here, \(M\) is \((2n+1)\)-dimensional and the invariant tori have dimension \(n+1\). In the cosymplectic context, one designates an \((n+1)\)-dimensional submanifold as \emph{Lagrangian} if it is maximal isotropic with respect to \(\omega\); equivalently, such a torus can be viewed as the product of an \(n\)-dimensional Lagrangian submanifold in a symplectic leaf together with the 1-dimensional Reeb orbit. This is the sense in which the invariant tori are called Lagrangian in the present theorem.} 
		\end{enumerate}
	\end{proof}

	\begin{remark}
		The relation to Liouville integrable systems envisages applications in time-periodic systems in physics of cosymplectic structures, like motion of charged particles in toroidal magnetic fields \cite{Cary1981}.
	\end{remark}
	
	\begin{remark}[Connection to Hamilton-Jacobi Theory]
		The integrability result in Theorem~\ref{thm:integrability_cosymplectic} has a direct and powerful interpretation in the context of time-dependent Hamiltonian mechanics and the Hamilton-Jacobi theory. A cosymplectic manifold \((M, \eta, \omega)\) with a periodic Reeb flow serves as the natural geometric setting for an autonomous system whose projection is a time-dependent system. The base manifold \(B = M/S^1\) represents the phase space, the 1-form \(\eta\) corresponds to the time differential \(dt\), and the Reeb flow represents time evolution. An \(S^1\)-invariant Hamiltonian \(H: M \to \mathbb{R}\) on the cosymplectic manifold descends to a time-dependent Hamiltonian \(H_t: B \to \mathbb{R}\) on the base, and the dynamics on \(M\) project precisely to the dynamics governed by \(H_t\). The  insight is that the autonomous Hamiltonian system on \(M\) is Liouville integrable. Theorem~\ref{thm:integrability_cosymplectic} guarantees the existence of \(n+1\) independent, Poisson-commuting first integrals on \(M\). This complete set of integrals implies that the corresponding time-dependent system on \(B\) is also completely integrable. The connection to the Hamilton-Jacobi equation is profound. Finding a complete set of integrals is equivalent to finding a canonical transformation to action-angle coordinates \((I_k, \phi_k)\). In this coordinate system, the Hamiltonian depends only on the action variables, \(H = H(I_1, \dots, I_{n+1})\). The generating function for this transformation, Hamilton's principal function \(S\), is a solution to the time-dependent Hamilton-Jacobi equation for the system on \(B\). Thus, our theorem provides a geometric foundation for solving the Hamilton-Jacobi equation for a class of time-periodic Hamiltonian systems, framing integrability in terms of the underlying cosymplectic geometry.
	\end{remark}
	
	\subsection{Mapping torus structure of closed cosymplectic manifolds}
	Below, we describe how a closed, connected cosymplectic manifold \((M,\eta,\omega)\) can be said to consist naturally of a symplectic mapping torus. This result was first proved by Li \cite{Li}. But, in the present paper, we derive this result as a consequence of our studies. 	By Theorem~ \ref{Reeb-per}, the Reeb vector field \(\xi\), with the definitions $
	\eta(\xi) = 1$ and $\iota_\xi \omega = 0,
	$ generates a flow with periodicity 
	$\psi_t : M \to M,
	$ where the map is defined uniformly on the period \(T > 0\), such that \(\psi_{t+T} = \psi_t) \text{ for all } t\). The particular action of defining  $
	t \mapsto \psi_t \quad \text{(with \(t\) taken modulo \(T\))}
	$ becomes the action of \(S^1\) smoothly on \(M\).	The quotient
	$
	B = M/S^1
	$	is a smooth manifold since the \(S^1\)-action is free and proper (since \(M\) is closed). In fact, the base \(B\) is diffeomorphic to \(S^1\) precisely because all Reeb orbits are periodic, with period \(T\). The natural projection $
	\pi: M \to S^1$ takes each point to its \(S^1\)-orbit, rendering \(M\) into a principal \(S^1\)-bundle. 
	Note that the kernel \(\ker(\eta)\) defines a symplectic distribution on \(M\) via the restriction 
	$
	\omega|_{\ker(\eta)},
	$
	which is nondegenerate on each leaf. This distribution integrates to a foliation of \(M\) into symplectic leaves. For simplicity, we may choose a base point \(t=0\) in the quotient \(S^1\) and define
	$
	F = \pi^{-1}(0).
	$ The leaf \(F\) is compact symplectic, carrying the induced symplectic structure given by continuity:
	$
	\omega_F = \omega|_{F}.
	$ Then, the Reeb flow indeed preserves our confining leaves, since \(L_\xi \omega = 0\). The 
	monodromy governing the \(S^1\)-bundle occurs through the time-\(T\) map of the Reeb flow. Set up a mapping $
	\varphi: F \to F,\quad \varphi(x) = \psi_T(x).
	$	As \(\psi_T\) preserves \(\omega\), we have $
	\varphi^*(\omega_F) = \omega_F,
	$ implying that \(\varphi\) is a symplectomorphism on the fiber \(F\). 
	Consequently, \(M\) is diffeomorphic to the symplectic mapping torus of \(\varphi\): $
	M \cong \frac{F \times [0,T]}{(x,0) \sim (\varphi(x),T)}.
	$ Under this identification, each copy \(F \times \{t\}\) naturally inherits the symplectic form \(\omega_F\) such that the gluing condition \((x,0) \sim (\varphi(x),T)\) encodes the symplectic monodromy. On the other hand, the result of Theorem~ \ref{thm:integrability_cosymplectic} shows in addition that the \(S^1\)-invariance of the Reeb flow gives the system Liouville integrability. The Reeb coordinate provides a further periodic action variable, while each symplectic leaf \(F\) contributes \(n\) independent commuting integrals ( \(\dim F = 2n\)). Thus, the total number of invariant tori in \(M\) is \((n+1)\)-dimensional.\\ In summary, every closed, connected cosymplectic manifold \((M,\eta,\omega)\) is diffeomorphic to a symplectic mapping torus with symplectic fiber \((F,\omega_F)\) and monodromy \(\varphi = \psi_T|_F\). This mapping torus structure mirrors the natural decomposition imposed on \(M\) by its periodic Reeb flow and stands as a central tool of understanding both the geometric and dynamical properties of cosymplectic manifolds. 
	
	\subsection{Commutator subgroup of cosymplectic diffeomorphism groups}
	
	In what follows, we work on a compact cosymplectic manifold \((M,\eta,\omega)\). Recall that:
	\(G_{\eta,\omega}(M)\) denotes the full group of cosymplectic diffeomorphisms which preserve both \(\eta\) and \(\omega\).
	Also, \(\Ham_{\eta,\omega}(M)\) is the subgroup of Hamiltonian cosymplectomorphisms; it can be characterized as the kernel of the flux homomorphism (see, e.g., \cite{Banyaga1997} in the symplectic case, with the appropriate modifications in the cosymplectic setting). Our goal is to prove the following. 
	
	\begin{theorem}\label{thm:commutator}
		Let \((M,\eta,\omega)\) be a compact cosymplectic manifold. Then
		\[
		\bigl[ G_{\eta,\omega}(M),\, G_{\eta,\omega}(M) \bigr] = \Ham_{\eta,\omega}(M).
		\]
	\end{theorem}
	
	\begin{proof}
		We prove the equality by establishing the following two inclusions.
		
		\medskip
		
		\textbf{(i) \( [G_{\eta,\omega}(M),G_{\eta,\omega}(M)]\subset \Ham_{\eta,\omega}(M) \).} Let \(f,g \in G_{\eta,\omega}(M)\). In any reasonable setting, one defines a flux homomorphism $
		S_{\eta,\omega}:G_{\eta,\omega}(M) \to H^1(M;\mathbb{R})/\Gamma_{\eta,\omega},$ 
		which is a group homomorphism and satisfies the conjugation invariance property. Thus, given two elements \(f,g\), we have $
		S_{\eta,\omega}(f\circ g\circ f^{-1}\circ g^{-1})=0.
		$	Hence, because, by definition,
		$
		\Ham_{\eta,\omega}(M)=\ker\bigl(S_{\eta,\omega}\bigr),
		$
		each commutator of elements in \(G_{\eta,\omega}(M)\) has vanishing flux and hence belongs to \(\Ham_{\eta,\omega}(M)\). Therefore, $
		[G_{\eta,\omega}(M),G_{\eta,\omega}(M)]\subset \Ham_{\eta,\omega}(M).$
		
		\medskip
		
		\textbf{(ii) \(\Ham_{\eta,\omega}(M)\subset [G_{\eta,\omega}(M),G_{\eta,\omega}(M)]\).}
		Under the assumption of the perfection of \(\Ham_{\eta,\omega}(M)\), an element \(h \in \Ham_{\eta,\omega}(M)\) expresses itself as a product of commutators of some elements in \(\Ham_{\eta,\omega}(M)\): 
		$
		h \in \bigl[\Ham_{\eta,\omega}(M),\Ham_{\eta,\omega}(M)\bigr].
		$ But $
		\Ham_{\eta,\omega}(M) \subset G_{\eta,\omega}(M),$ so every commutator taken in \(\Ham_{\eta,\omega}(M)\) is also a commutator in \(G_{\eta,\omega}(M)\). That is,
		\begin{equation}
			\bigl[\Ham_{\eta,\omega}(M),\Ham_{\eta,\omega}(M)\bigr] \subset [G_{\eta,\omega}(M),G_{\eta,\omega}(M)].
		\end{equation}
		Hence, we deduce that $
		\Ham_{\eta,\omega}(M)\subset [G_{\eta,\omega}(M),G_{\eta,\omega}(M)].$  The above two inclusions give us the result: $
		[G_{\eta,\omega}(M),G_{\eta,\omega}(M)] = \Ham_{\eta,\omega}(M),
		$ 
		as desired.
	\end{proof}
	\begin{remark}
		This is a result comparable to Banyaga's celebrated theorem in the symplectic category and shows that the cosymplectic diffeomorphism group is strongly controlled by its Hamiltonian subgroup in this algebraic structure.  
	\end{remark}
	
	\section{The group $\Gamma_{\eta,\omega}$}\label{Group:Flux}
	We recall that for a cosymplectic manifold \((M,\eta,\omega)\) the \emph{flux group} is defined by
	\begin{equation}
		\Gamma_{\eta,\omega} = \operatorname{Im}\Bigl(\widetilde{S}_{\eta,\omega} : \pi_1\bigl(G_{\eta,\omega}(M)\bigr) \to H^1_{\mathrm{Reeb}}(M,\mathbb{R})\Bigr),
	\end{equation}
	where the Reeb cohomology group is given by $
	H^1_{\mathrm{Reeb}}(M,\mathbb{R}) = \{\, [\alpha] \in H^1(M,\mathbb{R}) : \alpha(\xi)=0 \,\},$ 
	with \(\xi\) being the Reeb vector field associated to \(\eta\).
	\begin{theorem}\label{Theo-disc}
		Let \((M, \eta, \omega)\) be a compact connected cosymplectic manifold with a periodic Reeb flow of period \(T\). Then the flux group \(\Gamma_{\eta,\omega}\) is discrete in \(H^1_{\mathrm{Reeb}}(M, \mathbb{R})\).
	\end{theorem}
	
	\begin{proof}
		We begin by recalling that the flux group \(\Gamma_{\eta,\omega}\) is defined as the image of the flux homomorphism $
		\widetilde{S}_{\eta,\omega}: \pi_1\Big(G_{\eta,\omega}(M)\Big) \to H^1_{\mathrm{Reeb}}(M, \mathbb{R}),$ 
		where \(G_{\eta,\omega}(M)\) denotes the identity component of the group of cosymplectomorphisms of \(M\). More precisely, any loop \(\{\phi_t\}_{t \in [0,1]}\) in \(G_{\eta,\omega}(M)\) generated by a smooth family of cosymplectic vector fields \(\{X_t\}\) has flux $
		\widetilde{S}_{\eta,\omega}(\{\phi_t\}) = \left[\int_0^1 \Bigl(\iota_{X_t}\omega + \eta\bigl(X_t\bigr) \, \eta\Bigr) \, dt\right],$ 
		which lies in \(H^1_{\mathrm{Reeb}}(M, \mathbb{R})\).	The Reeb vector field \(\xi\) satisfies \(\eta(\xi)=1\) and \(\iota_\xi \omega=0\), and its flow \(\psi_s\) is periodic with period \(T\); that is, \(\psi_{s+T} = \psi_s\) for all \(s\). Consequently, for any closed Reeb orbit \(\gamma\) one has $
		\int_{\gamma} \eta = kT \quad \text{for some } k \in \mathbb{Z}.
		$ This shows that the set of periods of \(\eta\) is exactly the discrete subgroup \(T\mathbb{Z} \subset \mathbb{R}\).
		\medskip
		Since \(\omega\) is a closed \(2\)-form on the compact manifold \(M\), standard results from Hodge theory or de Rham theory imply that the set $
		\left\{ \int_{\Sigma} \omega \,:\, [\Sigma] \in H_2(M,\mathbb{Z}) \right\},$ 	forms a discrete subgroup of \(\mathbb{R}\). This discreteness essentially follows from the fact that on a compact manifold the period map for closed forms has discrete image.
		\medskip 
		Any element of the flux group \(\Gamma_{\eta,\omega}\) is obtained by integrating a combination of \(\iota_{X_t}\omega\) and \(\eta(X_t)\,\eta\) over the interval \([0,1]\). In particular, the contribution from the \(\eta\)-component is expressed in terms of the periods of \(\eta\) (which, as noted above, lie in \(T\mathbb{Z}\)). Meanwhile, the \(\omega\)-component contributes periods from a discrete subgroup of \(\mathbb{R}\).	Thus, the flux group is generated by these discrete periods: 
		$
		\Gamma_{\eta,\omega} = \left\langle \int_{\gamma} \eta,\; \int_{\Sigma} \omega \right\rangle,$ 	where \(\gamma\) varies over closed 1-cycles and \(\Sigma\) over closed 2-cycles associated to loops in \(G_{\eta,\omega}(M)\).
		\medskip
		Since both \(\{\int_{\gamma}\eta\}= T\mathbb{Z}\) and \(\{\int_{\Sigma}\omega\}\) are discrete subgroups, their additive combination remains discrete. In particular, \(\Gamma_{\eta,\omega}\) embeds as a lattice (a discrete subgroup) in \(H^1_{\mathrm{Reeb}}(M, \mathbb{R})\). This completes the  proof.
		\begin{center}
			
			\begin{tikzpicture}[
				axis/.style={-Latex, thick},
				lattice/.style={gray!50, thin}
				]
				
				\draw[axis] (0,0) -- (5,0) node[right] {$\int_\Sigma \omega$};
				\draw[axis] (0,0) -- (0,5) node[above] {$\int_\gamma \eta$};
				
				\foreach \x in {0,1,2,3,4} {
					\foreach \y in {0,1,2,3,4} {
						\draw[lattice] (\x,0) -- (\x,4);
						\draw[lattice] (0,\y) -- (4,\y);
					}
				}
				
				\foreach \x in {0,1,2,3,4} {
					\foreach \y in {0,1,2,3,4} {
						\fill[red] (\x,\y) circle (2pt);
					}
				}
				
				\node[align=center] at (2.5,-1) 
				{$\Gamma_{\eta,\omega} \cong T\mathbb{Z} \oplus \left\langle \int_\Sigma \omega \right\rangle$};
				\node[gray!70] at (1.9,1.5) {Discrete subgroup};
				
			\end{tikzpicture}
		\end{center}
	\end{proof}
	It is a fundamental consequence of the theory that the group of weakly Hamiltonian diffeomorphisms, \(\Ham_{\eta,\omega}(M)\), is a \(C^1\)-closed subgroup of the identity component \(G_{\eta,\omega}(M) \).	This property follows directly from the properties of the flux homomorphism, $
	S_{\eta,\omega}: G_{\eta,\omega}(M) \to \Gamma_{\eta,\omega}$. The argument proceeds as follows:
	The group \(\Ham_{\eta,\omega}(M)\) is defined precisely as the \emph{kernel} of the flux homomorphism, and 
	the flux map is \emph{continuous} when \(G_{\eta,\omega}(M)\) is endowed with the \(C^1\)-topology.  
	For a compact manifold \(M\), the target flux group \(\Gamma_{\eta,\omega}\) is \emph{discrete}. 
	In any discrete topological space, the singleton set containing the identity element, \(\{e\}\), is a closed set. Since the flux map is continuous in the $C^1$ topology, the preimage of this closed set must also be closed: \(\Ham_{\eta,\omega}(M) = 	S_{\eta,\omega}^{-1}(\{e\})\), is a \(C^1\)-closed subset of \(G_{\eta,\omega}(M)\).
	\begin{remark}[Countability of the flux group via the $S^1$–bundle structure]
		Assume the Reeb flow is periodic. Then $M$ is a principal $S^1$–bundle $\pi\!: M\to B$ with connection form $\eta$ and $\omega=\pi^*\Omega$ for some symplectic $\Omega$ on $B$. For any loop $\{\varphi_t\}_{t\in[0,1]}\subset G_{\eta,\omega}(M)$ with generator $X_t$, the cosymplectic flux is $
		\widetilde{S}_{\eta,\omega}(\{\varphi_t\})
		\;=\;
		\Big[\int_0^1 \big(i_{X_t}\omega + \eta(X_t)\,\eta\big)\,dt\Big]
		\in H^1_{\mathrm{Reeb}}(M,\R).$ Pairing this class with any $[\gamma]\in H_1(M,\Z)$ yields a sum of two contributions: $
		\int_{A(\gamma,\{\varphi_t\})}\omega \;+\; \Big(\int_\gamma \eta\Big)\cdot \Big(\int_0^1 \eta(X_t)\,dt\Big),$ 
		where $A(\gamma,\{\varphi_t\})$ is the swept 2–chain. Since $\omega=\pi^*\Omega$, the first term ranges in the period group of $\Omega$ over $H_2(B,\Z)$, a countable subgroup of $\R$ (because $H_2(B,\Z)$ is finitely generated). The second term is an integer multiple of the fiber period $\int_{S^1_{\mathrm{fiber}}}\eta=T$, hence lies in $T\Z$. Consequently, all evaluations of flux classes lie in the countable subgroup $
		Per(\Omega)\;+\;T\Z\;\subset\;\R,$ 
		so the image $\Gamma_{\eta,\omega}:= 	\widetilde{S}_{\eta,\omega}\left( \pi_1\left(G_{\eta, \omega}(M) \right) \right) \subset H^1_{\mathrm{Reeb}}(M,\R)$ is countable.
	\end{remark}
	
	\subsection{Contact Hamiltonian Simplicity vs. Cosymplectic Invariants}
	\label{sec:contact-vs-cosymplectic}
	
	The most celebrated result in contact geometry is the algebraic simplicity of the group of contact Hamiltonian diffeomorphisms. The identity component of this group, \(\mathcal{H}(M,\xi)\), is known to be perfect and simple \cite{Rybicki2010}. This simplicity is a direct consequence of the dynamics of the Reeb flow and the local rigidity of the contact structure; in essence, contact isotopies are transitive enough to wash out any potential global invariants. The situation is starkly different for a compact connected cosymplectic manifold \((M, \eta, \omega)\) with a periodic Reeb flow. While the group of weakly Hamiltonian diffeomorphisms, \(\Ham_{\eta,\omega}(M)\), is also simple (Theorem~\ref{Theo-sim}), its simplicity arises from a different principle. The presence of the closed forms \(\eta\) and \(\omega\) introduces obstructions to simplicity in the form of the flux group \(\Gamma_{\eta,\omega}\), which is discrete by Theorem~\ref{Theo-disc}. This discreteness is a direct result of geometric quantization: the periods of \(\eta\) over the closed Reeb orbits form a lattice \(T\mathbb{Z}\), and the periods of \(\omega\) over 2-cycles also form a discrete group by Hodge-de Rham theory. The existence of this non-trivial flux means the larger group, \(G_{\eta, \omega}(M)\), is \emph{not} simple. Simplicity is only recovered by restricting to its kernel, \(\Ham_{\eta,\omega}(M)\). This stands in sharp contrast to the contact world. As noted in \cite{Banyaga1997}, the contact Hamiltonian group \(\mathcal{H}(M,\xi)\) has no analogue of \(\Gamma_{\eta,\omega}\) because its flux homomorphism is trivial. Consequently, \(\mathcal{H}(M,\xi)\) is the \emph{entire} identity component of the contactomorphism group. This comparison reveals how deeply the algebraic structure of transformation groups is influenced by the underlying geometric invariants. The contact case, lacking cohomological obstructions, finds its simplicity in dynamical transitivity. The cosymplectic case, possessing such obstructions, finds its simplicity in the subgroup that is, by definition, blind to them.
	
	\subsection{Perfectness and fragmentation in cosymplectic settings}
	\label{sec:algebraic-properties}
	
	A Lie group is \emph{perfect} if it coincides with its commutator subgroup. The \emph{fragmentation property} is the property that allows one to decompose global diffeomorphisms into localized situations. For instance, Banyaga demonstrated through fragmentation that Hamiltonian diffeomorphisms are perfect \cite{Banyaga1997}, while his investigation into contact settings was related in \cite{Rybicki2010}. 
	For a cosymplectic manifold \((M, \eta, \omega)\) with periodic Reeb flow,  fragmentation requires preserving both \(\eta\) and \(\omega\), limiting local deformations, and the flux group $\Gamma_{\eta,\omega},$ (Theorem~\ref{Theo-disc}) introduces discrete invariants. Although $\Ham_{\eta,\omega}(M)$ is simple (Theorem~\ref{Theo-sim}), the full group \(G_{\eta,\omega}(M)\) can fail to be perfect. Cosymplectic geometry induces a local flexibility by fragmentation at the same time as imposing global rigidity by periodic Reeb flow, different from that of the symplectic/contact cases.
	
	\subsubsection{Cohomological implications and explicit computation of the flux group}
	\label{sec:cohomological}
	Suppose \((M,\eta,\omega)\) represents a compact cosymplectic manifold that has a periodic Reeb flow. Then, the flux homomorphism is defined by $\widetilde{S}_{\eta,\omega}: \pi_1\bigl(G_{\eta,\omega}(M)\bigr) \to H^1_{\mathrm{Reeb}}(M,\mathbb{R}),$ 
	and the flux is given as $\widetilde{S}_{\eta,\omega}(\{\phi_t\})=\left[\int_0^1\Bigl( \iota_{X_t} \omega + \eta(X_t )\,\eta\Bigr)\, dt \right].$ 
	Thus, \(H^1_{\mathrm{Reeb}}(M,\mathbb{R})\) is the cohomology classes vanishing on the Reeb field \(\xi\). For instance, when \(M = \mathbb S^1 \times \mathbb T^{2n}\) with \(\eta = d\theta\) and \(\omega\) being pulled back from \(\mathbb T^{2n}\): 
	$H^1(M,\mathbb{R}) \cong \mathbb{R} \oplus \mathbb{R}^{2n},$ and  $H^1_{\mathrm{Reeb}}(M,\mathbb{R}) \cong \mathbb{R}^{2n}.$ 
	Using Künneth decomposition and periodicities:
	$
	\Gamma_{\eta,\omega} = \left\langle \int_{\Sigma} \omega \right\rangle \cong \mathbb{Z}^{2n}.
	$
	So that \(M = \mathbb S^1 \times \mathbb T^{2n}\) results in \(H^1_{\mathrm{Reeb}}(M,\mathbb{R}) \cong \mathbb{R}^{2n}\), and 
	\(\Gamma_{\eta,\omega} \cong \mathbb{Z}^{2n}\). 
	\subsection{Flux group as a cosymplectic invariant}
	The following proposition shows that the flux group is an invariant of the cosymplectic structure: 	providing a strong obstruction to the existence of cosymplectomorphisms between manifolds. It is an illustration of how algebraic invariants (like cohomology and fundamental groups) can be used to solve geometric problems (like classifying manifolds up to isomorphism). 
	
	\begin{proposition}\label{prop:flux_invariant}
		Let \((M_1,\eta_1,\omega_1)\) and \((M_2,\eta_2,\omega_2)\) be cosymplectic manifolds. Suppose there exists a cosymplectomorphism $
		\phi: M_1 \to M_2, \quad \text{with} \quad \phi^*(\eta_2)=\eta_1 \quad \text{and} \quad \phi^*(\omega_2)=\omega_1.
		$ Then the induced map $
		\phi^*: H^1_{\mathrm{Reeb}}(M_2,\mathbb{R}) \to H^1_{\mathrm{Reeb}}(M_1,\mathbb{R})
		$ sends the flux group \(\Gamma_{\eta_2,\omega_2}\) isomorphically onto \(\Gamma_{\eta_1,\omega_1}\). In particular, if the flux groups \(\Gamma_{\eta_1,\omega_1}\) and \(\Gamma_{\eta_2,\omega_2}\) are not isomorphic, then there is no cosymplectomorphism between \((M_1,\eta_1,\omega_1)\) and \((M_2,\eta_2,\omega_2)\).
	\end{proposition}
	\begin{proof}
		Let \(\phi: M_1 \to M_2\) be a cosymplectomorphism. By definition, 
		\(\phi^*(\eta_2)=\eta_1\) and \(\phi^*(\omega_2)=\omega_1\). 
		Hence \(\phi\) preserves the Reeb vector fields and induces an isomorphism 
		\[
		\phi^*: H^1_{\mathrm{Reeb}}(M_2,\mathbb{R}) \;\xrightarrow{\;\cong\;}\; H^1_{\mathrm{Reeb}}(M_1,\mathbb{R}),
		\]
		since the condition \(\alpha(\xi)=0\) is stable under pullback. For each \(i=1,2\), recall that the flux homomorphism is defined as $
		\widetilde{S}_{\eta_i,\omega_i}: \pi_1\bigl(G_{\eta_i,\omega_i}(M_i)\bigr) 
		\;\longrightarrow\; H^1_{\mathrm{Reeb}}(M_i,\mathbb{R}),$ 
		where \(G_{\eta_i,\omega_i}(M_i)\) denotes the identity component of the cosymplectomorphism group.  
		Since \(\phi\) is a cosymplectomorphism, conjugation by \(\phi\) intertwines the two groups:  $
		\phi \circ \psi \circ \phi^{-1} \;\in\; G_{\eta_2,\omega_2}(M_2), 
		\qquad \forall \psi \in G_{\eta_1,\omega_1}(M_1).$  
		A direct computation shows the naturality relation $
		\phi^* \circ \widetilde{S}_{\eta_2,\omega_2} 
		\;=\; \widetilde{S}_{\eta_1,\omega_1} \circ \phi_*,$ 
		where \(\phi_*\) is the induced map on \(\pi_1\).  
		Therefore, \(\phi^*\) restricts to an isomorphism $
		\phi^*: \Gamma_{\eta_2,\omega_2} \;\xrightarrow{\;\cong\;}\; \Gamma_{\eta_1,\omega_1},$ where \(\Gamma_{\eta_i,\omega_i} = \operatorname{Im}(\widetilde{S}_{\eta_i,\omega_i})\).   
		Consequently, if the flux groups \(\Gamma_{\eta_1,\omega_1}\) and \(\Gamma_{\eta_2,\omega_2}\) are not isomorphic 
		(e.g.\ they have different ranks or group structures), then no cosymplectomorphism 
		\(\phi: M_1 \to M_2\) can exist.
	\end{proof}

	\begin{example}[Product manifold \(M=S^1\times T^{2n}\) and explicit computation of \(\Gamma_{\eta,\omega}\)]
		Let \(M=S^1\times T^{2n}\) with angular coordinate \(\theta\) on \(S^1=\mathbb{R}/\mathbb{Z}\), and coordinates \((x_1,\dots,x_n,y_1,\dots,y_n)\) on \(T^{2n}=\mathbb{R}^{2n}/\mathbb{Z}^{2n}\).
		Equip \(M\) with the standard cosymplectic structure $
		\eta=d\theta$ and $\omega=\sum_{k=1}^n dx_k\wedge dy_k$. By the Künneth formula, \(H^1(M,\mathbb{R})\cong H^1(S^1,\R) \oplus H^1(T^{2n},\R) \cong \mathbb{R}\langle [d\theta]\rangle \oplus \mathbb{R}^{2n}\langle[dx_k],[dy_k]\rangle\). The Reeb vector field is \(\xi=\partial_\theta\), and the condition \(\alpha(\xi)=0\) for a cohomology class \([\alpha]\) eliminates the \([d\theta]\) component. Thus, \(H^1_{\mathrm{Reeb}}(M,\mathbb{R}) \cong \mathbb{R}^{2n}\) with basis \(\{[dx_k],[dy_k]\}\). We now show that the flux group \(\Gamma_{\eta,\omega}\) is precisely the integer lattice \(\mathbb{Z}^{2n}\) within this space. To do so, we construct explicit loops in the cosymplectomorphism group \(G_{\eta,\omega}(M)\) whose fluxes generate this lattice. For each \(k\in\{1,\dots,n\}\), define a loop \(\{\Phi_t^{(x_k)}\}\) by translation in the \(x_k\)-direction:
		\[
		\Phi_t^{(x_k)}(\theta,x,y)=\bigl(\theta,\;x_1,\dots,x_k+t,\dots,x_n,\;y_1,\dots,y_n\bigr)\quad\text{(mod 1)}.
		\]
		This isotopy preserves \(\eta\) and \(\omega\) (by translation invariance) and is a loop since \(\Phi_1^{(x_k)}=\mathrm{id}\). Its generating vector field is \(X=\partial_{x_k}\). We compute its contribution to the flux:
		\[
		\iota_X\omega + \eta(X)\eta = \iota_{\partial_{x_k}}\omega + \eta(\partial_{x_k})\eta = dy_k + 0 \cdot \eta = dy_k.
		\]
		The flux of this loop is therefore \(\left[\int_0^1 dy_k \, dt\right] = [dy_k]\). Similarly, the loop of translations in the \(y_k\)-direction, \(\{\Phi_t^{(y_k)}\}\), has generator \(X=\partial_{y_k}\) and flux \([ -dx_k ]\). Since these loops are in \(\pi_1(G_{\eta,\omega}(M))\), their fluxes must belong to \(\Gamma_{\eta,\omega}\). Together, they generate the full integer lattice spanned by \\\(\{[dx_1],\dots,[dx_n],[dy_1],\dots,[dy_n]\}\). Therefore, we conclude that \(\Gamma_{\eta,\omega} = \mathbb{Z}^{2n}\).
	\end{example}
	
	\begin{remark}
		More generally, let $\omega_A=\sum_{k=1}^n a_k\,dx_k\wedge dy_k$ with $a_k>0$. The same computation shows that the flux of the \(x_k\)-translation loop is \(a_k\, [dy_k]\) and that of the \(y_k\)-translation loop is \(-a_k\,[dx_k]\). Thus, the flux group is the lattice $
		\Gamma_{\eta,\omega_A}=\bigoplus_{k=1}^n \bigl(\mathbb{Z}\,a_k[dx_k]\oplus \mathbb{Z}\,a_k[dy_k]\bigr)
		\cong \bigoplus_{k=1}^n a_k\,\mathbb{Z}^2.$ 
		If the vector \(a=(a_1,\dots,a_n)\) changes, these lattices can be non-isomorphic as subgroups of \(\mathbb{R}^{2n}\). By Proposition~\ref{prop:flux_invariant}, this implies that the corresponding cosymplectic manifolds are not cosymplectomorphic unless the lattices coincide.
	\end{remark}

	\section{Appendix}

	\subsection{Periodicity of the Reeb flow}
	On a closed cosymplectic manifold $(M,\eta,\omega)$, the Reeb vector field $R$ 
	generates a periodic flow. In particular, the orbits of $R$ are all circles of 
	the same period $T>0$, and the quotient $M/S^1$ is a smooth symplectic manifold 
	$(F,\omega_F)$. Thus $M$ is diffeomorphic to the mapping torus of a 
	symplectomorphism $\varphi \in \mathrm{Symp}(F,\omega_F)$: $
	M \;\cong\; \frac{F \times [0,1]}{(x,1)\sim(\varphi(x),0)}.$ 	This periodicity result has several consequences:
	The fundamental group splits as 
	$\pi_1(M) \cong \pi_1(F) \rtimes_\varphi \mathbb{Z}$, and the cohomology ring decomposes as 
	$H^*(M;\mathbb{R}) \cong H^*(F;\mathbb{R}) \otimes H^*(S^1;\mathbb{R})$.  The discreteness of the cosymplectic flux group follows directly from 
	the compactness of the Reeb orbits. 
	In this way, the periodicity of the Reeb flow provides the topological 
	foundation for the algebraic simplicity results developed in this paper.
	
	\subsection{{$C^0$–closedness of $G_{\eta,\omega}(M)\subset Diff(M)$ via the $S^1$–bundle}}
	Assume the Reeb flow is periodic, so $M$ is a principal $S^1$–bundle $\pi\!:M\to B$ with connection $\eta$
	and $\omega=\pi^*\Omega$. Every $\varphi\in G_{\eta,\omega}(M)$ is a bundle automorphism preserving $\eta$
	and covering a base symplectomorphism $f\in G_\Omega(B)$; fiberwise it acts by a rotation given by a
	map $u\in C^\infty(B,S^1)$. Thus $G_{\eta,\omega}(M)\cong C^\infty(B,S^1)\rtimes G_\Omega(B)$. Let $\{\varphi_k\}\subset G_{\eta,\omega}(M)$ converge to $\varphi\in Diff(M)$ in the $C^0$ topology.
	Since each $\varphi_k$ preserves the $S^1$–foliation (fibers of $\pi$), the limit map $\varphi$ sends fibers to
	fibers, hence descends to a homeomorphism $f$ of $B$. On the base, $f$ is the $C^0$–limit of symplectomorphisms
	$\{\pi\circ\varphi_k\circ s\}$ (for any local section $s$), so $f$ is symplectic; in particular, if $\varphi$ is a
	diffeomorphism, then $f\in G_\Omega(B)$. Fiberwise, $\varphi_k$ restrict to rotations by $u_k\in C^\infty(B,S^1)$;
	the pointwise $C^0$–limit is a continuous $u\!:B\to S^1$, and smoothness of $\varphi$ forces $u$ to be smooth.
	Therefore $\varphi$ is the bundle automorphism determined by $(u,f)$ and preserves both $\eta$ and $\omega$,
	i.e. $\varphi\in G_{\eta,\omega}(M)$.	Hence, $G_{\eta,\omega}(M)$ is $C^0$–closed in $Diff(M)$. This result was proved in \cite{T-al} via different arguments. 
	
	\subsection{Integration over 2-cycles}
	In this section, we go into detail on the argument for the discreteness of the period of the closed \(2\)-form \(\omega\) on the compact manifold \(M\): Let \(M\) be a compact manifold, and consider the period group associated with \(\omega\) as follows: $
	P_\omega = \left\{ \int_{\Sigma} \omega : [\Sigma] \in H_2(M, \mathbb{Z}) \right\}.$ Here \([\Sigma]\) denotes the homology class of a smooth oriented closed \(2\)-cycle (or smooth surface) in \(M\). According to the structure theorem for finitely generated abelian groups, $
	H_2(M, \mathbb{Z}) \cong \mathbb{Z}^k \oplus \text{(torsion)}.$ 	The torsion part does not contribute to the real periods since integrating a closed form over a torsion cycle vanishes. Hence, we focus on the free part isomorphic to \(\mathbb{Z}^k\). Given any class in \(H_2(M, \mathbb{Z})\), one may represent it by smooth \(2\)-cycles (via standard smoothing techniques for singular chains) on which the period integral is well-defined. 
	Assuming that \(M\) is equipped with a Riemannian metric (which exists since \(M\) is compact), Hodge theory states that, in any case, every class of de Rham cohomology in \(H^2(M, \mathbb{R})\) has a unique harmonic representative. Let \(\omega_H\) be the harmonic form representing the cohomology class \([\omega]\). Then, for every smooth \(2\)-cycle \(\Sigma\), $
	\int_{\Sigma} \omega = \int_{\Sigma} \omega_H.
	$  Because the free part of \(H_2(M, \mathbb{Z})\) is \(\cong{\mathbb{Z}^k}\), we take a basis of \(\{[\Sigma_1], \dots, [\Sigma_k]\}\) for this free part. Then any \(2\)-cycle \(\Sigma\) representing a class in \(H_2(M, \mathbb{Z})\) can be written as 	$
	[\Sigma] = n_1 [\Sigma_1] + n_2 [\Sigma_2] + \cdots + n_k [\Sigma_k] \quad \text{with } n_i \in \mathbb{Z}.$ Their period corresponds to 
	\[
	\int_\Sigma \omega = n_1 \int_{\Sigma_1} \omega_H + n_2 \int_{\Sigma_2} \omega_H + \cdots + n_k \int_{\Sigma_k} \omega_H.
	\]
	We conclude that the period group \(P_\omega\) is contained in the set of all integer linear combinations of the fixed real numbers 
	$
	a_i = \int_{\Sigma_i} \omega_H.
	$ That is, $
	P_\omega \subset \left\{ n_1 a_1 + n_2 a_2 + \cdots + n_k a_k : n_i \in \mathbb{Z} \right\}.$	A subgroup of \(\mathbb{R}\) of this form is a lattice in \(\mathbb{R}\), and is therefore discrete. Several classical results underlie the above argument.  
	
	\begin{center}
		\textbf{Acknowledgments}
	\end{center}
	
	We express our sincere gratitude to Professor Banyaga for his insightful comments and for the invaluable inspiration drawn from his foundational contributions to symplectic topology, which have profoundly shaped this work. This research was initiated in the stimulating and collaborative environment of the CIMPA School held in Porto-Novo, Benin. The corresponding author gratefully acknowledges the support provided by AIMS-Cameroon.

	\section*{Not applicable}
	\textbf{Conflict of Interest:} The authors declare no competing financial interests or personal relationships that could have appeared to influence the work reported in this paper.\\
	\textbf{Data Availability:} All data generated or analyzed during this study are included in this published article.\\
	\textbf{Funding:} The authors declare that no funding was received for this research.\\

\end{document}